\documentclass{article}
\usepackage[utf8]{inputenc}
\usepackage[letterpaper, portrait, margin=1in]{geometry}
\usepackage{amsmath, amssymb, amsthm}
\usepackage{graphicx}
\usepackage{mathrsfs}
\usepackage{xcolor}
\usepackage{bbm}
\usepackage{mathtools}
\usepackage[linesnumbered,lined,ruled]{algorithm2e}
\usepackage{algpseudocode}
\usepackage{chngcntr}
\usepackage{bbm}
\usepackage[linesnumbered,lined,ruled]{algorithm2e}
\usepackage{mathrsfs}  
\usepackage{float,stmaryrd}
\usepackage{float}
\counterwithin{figure}{section}
\graphicspath{{images/}}
\theoremstyle{definition} 
\newtheorem{theorem}{Theorem}[section]
\newtheorem{corollary}[theorem]{Corollary}
\newtheorem{coro}[theorem]{Corollary}
\newtheorem{lemma}[theorem]{Lemma}
\newtheorem{proposition}[theorem]{Proposition}
\newtheorem{prop}[theorem]{Proposition}

\def\cf{cf.\kern.3em}
\def\eg{e.g.\kern.3em}
\def\ie{i.e.\kern.3em}
\def\resp{\text{resp.}\kern.3em}
\def\etc{etc.\kern.3em}

\title{Graphs of Gonality Three}
\author{Ivan Aidun, Frances Dean, Ralph Morrison, Teresa Yu, Julie Yuan}
\date{}

\begin{document}

\maketitle

\begin{abstract}    In 2013, Chan classified all metric hyperelliptic graphs, proving that divisorial gonality and geometric gonality are equivalent in the hyperelliptic case. We show that such a classification extends to combinatorial graphs of divisorial gonality three, under certain edge- and vertex-connectivity assumptions. We also give a construction for graphs of divisorial gonality three, and provide conditions for determining when a graph is not of divisorial gonality three.
\end{abstract}

\noindent \textbf{Note:  the original published version of this paper had an error in the statement of the main theorem; a correction of this will appear as a corrigendum published by the same journal.  The corrections have been incorporated into this arXiv version.}

\section{Introduction}  \setcounter{section}{1}

  Tropical geometry studies graphs as discrete analogues of algebraic curves.  A motivating goal of this program is to prove theorems in algebraic geometry using combinatorial methods, as in \cite{cdpr}.  In \cite{bn}, Baker and Norine define a theory of divisors on combinatorial graphs similar to divisor theory on curves, proving a Riemann-Roch type theorem. This was extended by \cite{gk} and \cite{mz} to metric graphs, which have lengths associated to each edge. To model maps between curves, harmonic morphisms between simple graphs were introduced in \cite{ura}, extended to multigraphs in \cite{bn2}, and finally to metric graphs in \cite{chan}.

An important invariant of an algebraic curve is its \emph{gonality}. This is the minimum degree of a divisor of rank $1$, or equivalently, the minimum degree of a morphism from the curve to a line \cite[Section 8C]{syzygies}. We can extend these definitions to combinatorial and metric graphs, using either divisor theory or morphisms from the graph to a tree. However, unlike in classical algebraic geometry, these two notions of gonality defined on graphs are in general inequivalent, as demonstrated in \cite{liyau}. We thus define two different types of gonality: \emph{divisorial gonality} and \emph{geometric gonality}. (Whenever we refer to the \emph{gonality} of a graph without specifying which type, we mean the divisorial gonality.)

Our two notions of gonality happen to agree when either is equal to $1$: divisorial gonality is equal to $1$ if and only if the graph is a tree, and the same is true of geometric gonality \cite[Lemma 1.1 and Example 3.3]{bn2}. This no longer holds when our graph has higher divisorial gonality; for example, the \emph{banana graph}, which has two vertices and $n\geq 2$ edges connecting the two vertices, has divisorial gonality $2$ and geometric gonality $n$ \cite{liyau}. However, this turns out to be the only such example for graphs of divisorial gonality $2$, as shown by the following result.

\begin{theorem}[Theorem 1.3 in \cite{chan}, slightly reordered]\label{thm:chan}
Let $\Gamma$ be a metric graph with no points of valence $1$ and canonical loopless model $(G,\ell)$. Then the following are equivalent:
\begin{enumerate}
    \item $G$ has \textup{(}divisorial\textup{)} gonality $2$.
    \item  There exists a non-degenerate harmonic morphism $\varphi : G \to T$ where $\text{deg}(\varphi) = 2$ and $T$ is a tree, or $|V(G)|=2$.
    \item There exists an involution $i:G\rightarrow G$ such that $G/i$ is a tree.
\end{enumerate}
\end{theorem}

Note that the only (connected) graphs $G$ with $|V(G)| = 2$ are those belonging to the family of banana graphs. Hence, Theorem \ref{thm:chan} implies that, for all other metric graphs, having divisorial gonality $2$ and having geometric gonality $2$ are equivalent.

Our main result in this paper is an analogue of Theorem \ref{thm:chan} for graphs of divisorial gonality $3$. Although Theorem \ref{thm:chan} is stated for metric graphs, ours holds only for combinatorial graphs, without the data of lengths associated to the edges.  For one portion we need an additional assumption on our graph $G$ called the \emph{zero-three condition}, defined at the beginning of Section \ref{section:simple}.

\begin{theorem}
\label{thm:maintheorem}
Consider the following $3$ conditions:
\begin{enumerate}
    \item $G$ has \textup{(}divisorial\textup{)} gonality $3$.
    \item There exists a non-degenerate harmonic morphism $\varphi : G \to T$ where $\text{deg}(\varphi) = 3$ and $T$ is a tree.
        \item There exists a cyclic automorphism $\sigma: G \to G$ of order $3$ that does not fix any edge of $G$ satisfying the property that $G/\sigma$ is a tree.
\end{enumerate}
If $G$ is $3$-edge-connected, then (1) and (2) are equivalent.
Moreover, if $G$ is simple and $3$-vertex-connected, then condition (3) implies conditions (1) and (2); and if $G$ satisfies the zero-three condition, then (1) implies (3).
\end{theorem}

The decision to restrict our attention to $3$-vertex-connected graphs in the simple case is in part supported by Proposition \ref{prop:twoconnect}, which shows that a simple, bridgeless, trivalent graph that is not $3$-vertex-connected must have gonality at least $4$. Moreover, the example graph in Figure \ref{fig:no_automorphism} shows that $3$-edge-connectedness is not a strong enough assumption to guarantee the existence of a cyclic automorphism of order $3$; and the example graph in Figure \ref{fig:wheel} shows that the zero-three condition is necessary for (1) implies (3), even for $3$-vertex-connected graphs. To justify our $3$-edge-connected assumption for the multigraph case, we point to recent work by Corry and Steiner, appearing in \cite[Theorem 10.24]{cp}, which shows that for $d$-edge-connected graphs with more than $d$ vertices, the set of degree $d$ non-degenerate harmonic morphisms to a tree is in bijection with divisors of degree $d$ and rank $1$ on the graph. With some extra work to rule out the possibility of hyperellipticity, this result can be used to prove our main theorem for multigraphs. However, the proof we present is independently formulated.

Our paper is organized as follows. In Section \ref{section:definitions}, we establish definitions and notation, and review previous results on divisors and harmonic morphisms of graphs. In Section \ref{section:multi}, we prove Theorem \ref{thm:multigraph}, which is the first part of Theorem \ref{thm:maintheorem} and applies to general multigraphs. In Section \ref{section:simple}, we restrict our attention to simple graphs in order to prove Theorem \ref{thm:simpletheorem}, which adds the third condition in Theorem \ref{thm:maintheorem}. We also give a criterion for identifying graphs with gonality strictly greater than $3$.  Finally, in Section \ref{section:constructions}, we present a construction for a (proper) subset of graphs of gonality $3$.

\section{Definitions and Notation}
\label{section:definitions}

We define a \textit{graph} $G = (V, E)$ with vertex set $V(G)$ and edge set $E(G)$ to be a finite, connected, loopless, multigraph. Throughout this paper, all graphs are assumed to be combinatorial (that is, without lengths assigned to edges) unless otherwise stated. Graphs with no multiedges are called \textit{simple}. Given a vertex $v \in V(G)$ and an edge $e \in E(G)$, we use the notation $v \in e$ to indicate that $v$ is an endpoint of $e$. For $u,v \in V(G)$, define $E(u,v) \coloneqq \{ e \in E(G) : u \in e, v \in e\}$.  Similarly, for $A, B \subset V(G)$ define $E(A, B) \coloneqq \{ e \in E(G) : e \in E(a, b) \; \text{for some} \; a \in A, b \in B\}$. The \textit{valence} of a vertex $v \in V(G)$ is defined as $\text{val}(v) \coloneqq |\{ e \in E(G) : v \in e \}|$. We define the \textit{genus} of a graph $G$ as $g(G) \coloneqq |E(G)| - |V(G)| + 1$. A graph of genus 0 is called a \textit{tree}.

A graph $G = (V, E)$ is \textit{$k$-edge-connected} if, for any set $W$ of $k-1$ edges, the subgraph $(V,E-W)$ is connected. We let $\eta(G)$ denote the \textit{edge-connectivity} of the graph. That is, $\eta(G)$ is the maximum integer $k$ such that $G$ is $k$-edge-connected. A \textit{bridge} of $G$ is an edge whose deletion strictly increases the number of connected components of $G$. A graph is \textit{bridgeless} if it has no bridges, or equivalently if it is $2$-edge-connected.

Similarly, a graph $G = (V, E)$ is \textit{$k$-vertex-connected} (or just \textit{$k$-connected}) if, for any set $U$ of $k-1$ vertices, the subgraph $(V-U,E)$ is connected. We let $\kappa(G)$ denote the \textit{vertex-connectivity} of the graph. That is, $\kappa(G)$ is the maximum integer $k$ such that $G$ is $k$-vertex-connected.  (By convention, we set $\kappa(G)=|V|-1$ if every pair of vertices in $G$ is joined by an edge.) Since removing a vertex from a graph removes all edges incident to that vertex, we have that $\kappa(G)\leq \eta(G)$ for any graph $G$.

\subsection{Divisor Theory on Graphs}

We now review the key concepts of divisor theory on graphs, as developed in \cite{baker}. A \textit{divisor} $D$ on a graph $G$ is a $\mathbb{Z}$-linear combination of vertices. We will often explicitly write out divisors with the notation
\[
D = \sum_{v \in V(G)} D(v) \cdot (v),
\]
where $D(v)$ denotes the value of $D$ at $v$.
The set of all divisors $\text{Div}(G)$ on a graph $G$ forms an abelian group under component-wise addition. The \textit{degree} of a divisor $D$ is defined as the sum of its integer coefficients:

\[
\text{deg}(D) \coloneqq \sum_{v \in V(G)} D(v).
\]

For a fixed $k \in \mathbb{Z}$, let $\text{Div}^k(G)$ be the set of all divisors of degree $k$ on $G$. A divisor $D$ is \textit{effective} if, for all $v \in V(G)$, $D(v) \geq 0$. Let $\text{Div}_+(G)$ be the set of all effective divisors on a graph $G$ and for $k \in \mathbb{Z}_{>0}$, let $\text{Div}_+^k(G)$ be the set of all effective divisors of degree $k$ on $G$. For a given effective divisor $D$, we define the \textit{support} of $D$ as
\[
\text{supp}(D) \coloneqq \{ v \in V(G) : D(v) > 0 \}.
\]

The \textit{Laplacian} $\mathcal{L}(G)$ of a graph $G$ is the $|V|\times |V|$ matrix with entries

\[
\mathcal{L}_{v,w} = 
\begin{cases}
\text{val}(v) \quad &\text{if $v = w$} \\
-| E(v, w) | \quad &\text{if $v \neq w$}.
\end{cases}
\]
We use $\Delta : \text{Div}(G) \to \text{Div}(G)$ to denote the \textit{Laplace operator} associated with the Laplacian matrix. A \textit{principal divisor} is a divisor in the image of $\Delta$. We use $\text{Prin}(G)$ to denote the set of principal divisors on a graph $G$, \ie $\text{Prin}(G) = \Delta(\text{Div}(G))$. Notice that $\text{Prin}(G)$ is a normal subgroup of $\text{Div}^0(G)$. We can therefore define the \textit{Jacobian} $\text{Jac}(G)$ of a graph $G$ as the quotient group $\text{Div}^0(G) / \text{Prin}(G)$.

Now, define an equivalence relation $\sim$ on divisors such that $D \sim E$ if and only if $D - E \in \text{Prin}(G)$. We say in this case that $D$ and $E$ are \textit{linearly equivalent} and define the \textit{linear system} associated with a divisor $D$ as
\[
|D| \coloneqq \{ E \in \text{Div}_+(G) : E \sim D \}.
\]
For a divisor $D \in \text{Div}(G)$, we define the \textit{rank} of $D$ as $r(D) \coloneqq -1$ if $|D| = \varnothing$, and otherwise as
\[
r(D) \coloneqq \max \{ k \in \mathbb{Z} : |D - F| \neq \varnothing \; \text{for all} \; F \in \text{Div}_+^k(G) \}.
\]
The \textit{gonality} of a graph $G$ is defined as
\[
\text{gon}(G) \coloneqq \min \{ \text{deg}(D) : D \in \text{Div}_+(G), r(D) \geq 1 \}.
\]  Later, when we need to distinguish between two different types of gonality, this will be referred to as \emph{divisorial gonality}.

\subsection{Baker-Norine Chip-Firing}

The definition of gonality provided in the previous section has an equivalent statement in terms of chip-firing games on graphs. In a chip-firing game, we think about placing integer numbers of poker chips on the vertices of our graph.  A negative number of poker chips corresponds to a vertex being ``in debt''.  A \emph{chip-firing move} involves selecting a vertex $v \in V(G)$, subtracting $\text{val}(v)$ chips from $v$, and adding $|E(v,v')|$ chips to each $v'$ adjacent to $v$.

The \textit{Baker-Norine chip-firing game} is played with the following rules:

\begin{enumerate}
    \item One player places $k$ chips on the vertices $V(G)$ of a graph $G$ in any arrangement.
    \item Another player (the adversary) chooses a vertex $v \in V(G)$ from which to subtract a chip, possibly putting $v$ into debt.
    \item The first player wins if they can reach a configuration of chips where no vertex is in debt via a sequence of chip-firing moves.  Otherwise, the adversary wins.
\end{enumerate}

Notice that these ``chip configurations'' correspond to divisors on graphs. By standard results as in \cite{bn}, chip-firing moves correspond to subtracting principal divisors; the divisors present before and after chip-firing are equivalent; and the gonality of a graph is equivalent to the minimum number of chips $k$ required to guarantee that the first player has a winning strategy in the Baker-Norine chip-firing game. Hence, we define a \textit{winning divisor} $D$ to be a divisor satisfying $r(D) \geq 1$. 

Since chip-firing is a commutative operation, we can chip-fire from an entire subset $A \subset V(G)$ at once by sending a chip along each edge outgoing from the subset. Let $\mathbbm{1}_A$ denote the indicator vector on $A$. Then, given a divisor $D$, the resulting divisor after chip-firing from the subset $A$ is $D - \Delta \mathbbm{1}_A$. We define the \textit{outdegree} of $A$ from a vertex $v \in A$ to be the number of edges leaving $A$ from $v$, so
\[
\text{outdeg}_v(A) \coloneqq | E(\{v\}, V(G) - A) |.
\]
Hence, a chip-firing move from a subset $A \subset V(G)$ sends $\text{outdeg}_v(A)$ chips from each vertex $v \in A$ into $V(G) - A$. The total outdegree of $A$ is defined as $\text{outdeg}_A(A) \coloneqq \sum_{v \in A} \text{outdeg}_v(A)$. The following result is proven in \cite{db}.

\begin{lemma}
\label{lemma:finiteseq}
Given an effective divisor $D$ and an equivalent effective divisor $D'$, there exists a finite sequence of subset-firing moves which transforms $D$ into $D'$ without ever inducing debt in any vertex of the graph. 
\end{lemma}

This means that if we have a divisor $D$ with $r(D) \geq 1$, then we can move at least one chip onto every vertex of our graph (in turn) without ever putting any of the vertices of the graph into debt. For a given divisor $D$, we say $D$ is \textit{$v$-reduced} with respect to some vertex $v \in V(G)$ if
\begin{enumerate}
    \item for each $v' \in V(G) - \{v\}$, $D(v') \geq 0$, and
    \item for any nonempty subset $A \subset V(G) - \{v\}$, there exists $v' \in A$ such that $\text{outdeg}_{v'}(A) < D(v')$.
\end{enumerate}
This means that every vertex (except possibly $v$) is out of debt, and that there exists no way to fire from any subset of $V(G) - \{v\}$ without inducing debt. The following two results are proven in \cite{bn}:

\begin{lemma}
\label{lemma:uniquevred}
Given a divisor $D \in \text{Div}(G)$ and a vertex $v \in V(G)$, there exists a unique $v$-reduced divisor $D'$ such that $D' \sim D$. 
\end{lemma}

We will use $\text{Red}_{v}(D)$ to denote this unique $v$-reduced divisor. 

\begin{lemma}
\label{lemma:vredrank}
For a divisor $D \in \text{Div}(G)$, $r(D) \geq 1$ if and only if $\text{Red}_v(D)(v) \geq 1$ for each $v \in V(G)$. 
\end{lemma}

Thus, we can determine if a divisor is winning divisor by checking that, for each $v \in V(G)$, the associated $v$-reduced divisor satisfies $v \in \text{supp}(\text{Red}_v(D))$. Furthermore, given a divisor $D$ and a vertex $v$ for which $D$ is effective away from $v$, Algorithm \ref{algo:dhars}, developed by Dhar in \cite{dhar}, computes $\text{Red}_v(D)$.

\begin{algorithm}
	\KwData{Graph $G$, vertex $v \in V(G)$, and divisor $D \in \text{Div}(G)$, $D(v') \geq 0$ for $v' \neq v$}
	\KwResult{$\text{Red}_v(D)$}
	$V' \coloneqq \{ v \}$, $E' \coloneqq \{ e \in E(G) : v \in e \}$\;
	\While{$V' \neq V(G)$} {
	    \If{$D(v') < |\{e \in E' : v' \in e\}|$ for some $v' \in V(G)$} {
	    $V' = V' + \{ v'\}$, $E' = E' + \{ e \in E(G) : v' \in e\}$\;
	    }
	    \Else {
	        return $\textbf{Alg}(G, v, D - \Delta \mathbbm{1}_{V(G) - V'})$\;
	    }
	}
	return $D$\;
\caption{Dhar's Burning Algorithm}
\label{algo:dhars}
\end{algorithm}

We offer the following intuitive explanation of Algorithm \ref{algo:dhars}. We begin with a graph $G$, a vertex $v$, and a divisor $D$, which is assumed to be effective away from $v$. Then we ``start a fire'' at the vertex $v$. 
As the fire spreads through the graph, chips on vertices act as ``firefighters'', protecting their vertex from the encroaching flames. To determine which vertices and edges of the graph catch on fire, we repeat the following two steps until no new vertices or edges burn.
\begin{enumerate}
    \item If an edge is adjacent to a burning vertex, then that edge also catches fire and begins to burn.
    \item If a vertex is adjacent to more burning edges than it has chips, then that vertex begins to burn.
\end{enumerate}
Once a stable state is reached, we chip-fire from the set of unburnt vertices. Then we begin the burning process again starting at $v$. If at any point the entire graph burns, the algorithm terminates and outputs the resulting divisor.

We refer the reader to \cite{bs} for a proof that Algorithm \ref{algo:dhars} terminates and that the resulting divisor is indeed $\text{Red}_v(D)$. As a corollary of Lemma \ref{lemma:vredrank}, we have the following result.

\begin{corollary}
\label{lemma:dhars}
For an effective divisor $D \in \text{Div}_+(G)$, if there exists some $v \in V(G)$ such that $v \notin \text{supp}(D)$ and for which beginning Dhar's burning algorithm at $v$ results in the entire graph burning, then $r(D) < 1$.
\end{corollary}

\subsection{Riemann-Roch for Graphs}

For a graph $G$, we define the \textit{canonical divisor} as
\[
K \coloneqq \sum_{v \in V(G)} (\text{val}(v) - 2) (v).
\]
The canonical divisor has degree $2g(G) - 2$. In \cite{bn}, Baker and Norine prove the following Riemann-Roch theorem for graphs, analogous to the classical Riemann-Roch theorem on algebraic curves:

\begin{theorem}[Riemann-Roch for graphs]
\label{thm:riemannroch}
If $G$ is a graph with $D \in \text{Div}(G)$, 
\[
r(D) - r(K-D) = \text{deg}(D) + 1 - g(G).
\]
\end{theorem}

Notice that this implies $r(K) = g(G) - 1$. As a consequence, we can prove the following result:

\begin{proposition}
\label{prop:genus2}
If $G$ is a graph with genus $g(G) \leq 2$, then $\text{gon}(G) \leq 2$.
\end{proposition}

\begin{proof}
If $g(G) = 0$, then $G$ must be a tree, giving $\text{gon}(G) = 1$. If $g(G) = 1$ and $D \in \text{Div}(G)$ satisfies $\text{deg}(D) = 2$, then by Riemann-Roch for graphs, we see that
\begin{align*}
    r(D) &= \text{deg}(D) + 1 - g + r(K - D) \\
    &= 2 + r(K - D) = 2+(-1)=1,
\end{align*}
where $r(K-D)=-1$ since $\deg(K-D)<0$.
Finally, if $g(G) = 2$, then the canonical divisor $K$ has $\text{deg}(K) = 2$ and $r(D) = 1$, providing an upper bound on the gonality of $G$.
\end{proof}

\subsection{Harmonic Morphisms of Graphs}

We now turn to another notion of gonality called \textit{geometric gonality}, which is defined in terms of maps between graphs. If $G$ and $G'$ are graphs, a \textit{morphism} $\varphi : G \to G'$ is a map sending $V(G) \to V(G')$ and $E(G) \to E(G') \cup V(G')$, satisfying the following two conditions:

\begin{enumerate}
    \item if $e = uv \in E(G)$ and $\varphi(u) = \varphi(v)$, then $\varphi(e) = \varphi(u) = \varphi(v)$
    \item if $e = uv \in E(G)$ and $\varphi(u) \neq \varphi(v)$, then $\varphi(e) = \varphi(u)\varphi(v)$.
\end{enumerate}

This definition comes from \cite{bn2}. Morphisms defined on graphs are sometimes \emph{indexed}, as in \cite{liyau}. In this paper, we will only consider non-indexed morphisms. For a vertex $v \in V(G)$, we define the \textit{multiplicity} of $\varphi$ at $v$ with respect to an edge $e' \ni \varphi(v)$ as
\[
m_\varphi(v) \coloneqq | \{ e \in E(G) : v \in e, \varphi(e) = e' \} |,
\]
for some choice of $e' \in E(G')$ incident with $\varphi(v)$. A morphism is \textit{harmonic} if the value of $m_\varphi(v)$ does not depend on the choice of $e' \in E(G')$. A harmonic morphism is \textit{non-degenerate} if $m_\varphi(v) > 0$ for all $v \in V(G)$. We define the \textit{degree} of a harmonic morphism to be
\[
\text{deg}(\varphi) \coloneqq | \{ e \in E(G) : \varphi(e) = e' \} | = |\varphi^{-1}(e')|,
\]
for some choice of $e' \in E(G')$. The degree of a harmonic morphism is well-defined and independent of the choice of $e'$ \cite[Lemma 2.4]{bn2}. Figure \ref{fig:morphismex} depicts an example of a non-degenerate harmonic morphism. Notice that for each edge $e \in T = \varphi(G)$, we have $|\varphi^{-1}(e)| = 3$.

\begin{figure}
    \centering
    \includegraphics[scale=0.2]{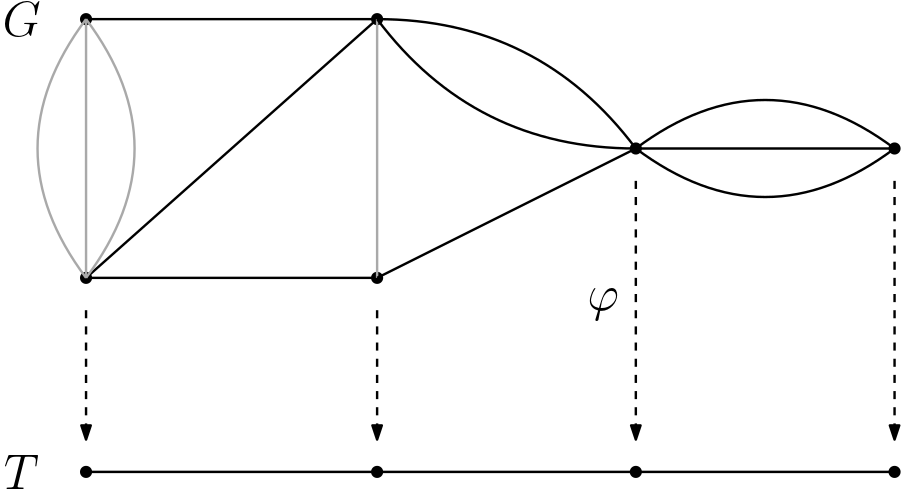}
    \caption{A non-degenerate harmonic morphism of degree $3$ from $G \to T$}
    \label{fig:morphismex}
\end{figure}

We define the \textit{geometric gonality} of a graph $G$ to be
\[
\text{ggon}(G) \coloneqq \min \{ \text{deg}(\varphi) : \varphi : G \to T \; \text{is a non-degenerate harmonic morphism onto a tree} \; T \}.
\]

We remark that there are multiple inequivalent notions of geometric gonality defined in the literature. In particular, some authors consider \textit{refinements} of the original graph \cite{cd}, while other authors only consider graph morphisms that are also \textit{homomorphisms} \cite[Section 1.3]{vdw}. The results in our paper hold specifically for the definition of geometric gonality given above.

\subsection{Bounds on Gonality}

The following result is stated in \cite{liyau} and proven here for the reader's convenience.

\begin{lemma}
\label{lemma:eta}
For a graph $G$, $\text{gon}(G) \geq \min \{|V(G)|, \eta(G)\}$.
\end{lemma}

\begin{proof}
Suppose that $D \in \text{Div}_+(G)$ is a divisor with $\text{deg}(D) < \min \{ |V(G)|, \eta(G) \}$. This means that $D$ does not contain all of the vertices of $G$ in its support, nor can we fire from any subset of $\text{supp}(D)$ because any such subset $A \subseteq \text{supp}(D)$ will have $\text{outdeg}_A(A) > \sum_{v \in A} D(v)$. Hence, $D$ is not a winning divisor.
\end{proof}

The \textit{treewidth} $\text{tw}(G)$ of a graph $G$ is defined to be the minimum width amongst all possible tree decompositions of $G$. The following result is proven in \cite{dbg}.

\begin{lemma}
\label{ref:treewidth}
For a graph $G$, $\text{gon}(G) \geq \text{tw}(G)$.
\end{lemma}

It is shown in \cite{two} that, for a simple graph $G$, $\text{tw}(G) \geq \min \{\text{val}(v) : v \in V(G) \}$. Hence, we have the following result.

\begin{lemma}
\label{lemma:minvalence}
For a simple graph $G$, $\text{gon}(G) \geq \min \{ \text{val}(v) : v \in V(G) \}$.
\end{lemma}

We also have the following ``trivial'' upper bound on gonality.

\begin{lemma}
\label{lemma:upperbound}
For a graph $G$, $\text{gon}(G) \leq |V(G)|$.
\end{lemma}

This upper bound is typically only attained when the edge-connectivity of the graph is high relative to the number of vertices. In fact, if $G$ has a vertex $v$ which is not incident to any multiple edges, then $\text{gon}(G) \leq |V(G)|-1$, since placing one chip on every vertex except $v$ results in a winning divisor.

\section{Multigraphs of Gonality Three}
\label{section:multi}

In this section, we will prove the following result, which is simply the first part of Theorem \ref{thm:maintheorem} and applies to all multigraphs of edge-connectivity at least $3$. 

\begin{theorem}
\label{thm:multigraph}
If $G$ is a $3$-edge-connected graph, then the following are equivalent:
\begin{enumerate}
    \item $G$ has gonality $3$.
    \item There exists a non-degenerate harmonic morphism $\varphi : G \to T$ where $\text{deg}(\varphi) = 3$ and $T$ is a tree, or $|V(G)| = 3$.
\end{enumerate}
\end{theorem}

One approach to proving this result would be to apply \cite[Theorem 10.24]{cp} with $d=3$ for the direction (1) implies (2), and then argue that if (2) holds then hyperellipticity is not possible.  Here we present a different approach, which will also lay the groundwork for the subsequent section. We will first prove some preliminary results, which will allow us to define an equivalence relation on the vertices of $G$. From here, the map from $G$ to the resulting quotient graph provides our non-degenerate harmonic morphism.

\begin{lemma}
\label{prop:allgon3}
If $G$ is a graph with $\text{gon}(G) = 3$, then
\begin{enumerate}
    \item $g(G) \geq 3$, and
    \item either $G$ is at most $3$-edge-connected, or $|V(G)| = 3$.
\end{enumerate}
\end{lemma}

\begin{proof}
Note that (1) comes as a corollary of Proposition \ref{prop:genus2}. For (2), assume that $|V(G)| \geq 4$ and $\eta(G) \geq 4$. Then, by Lemma \ref{lemma:eta}, we have $\text{gon}(G) \geq 4$.

\begin{figure}[ht]
    \centering
    \includegraphics[scale=0.2]{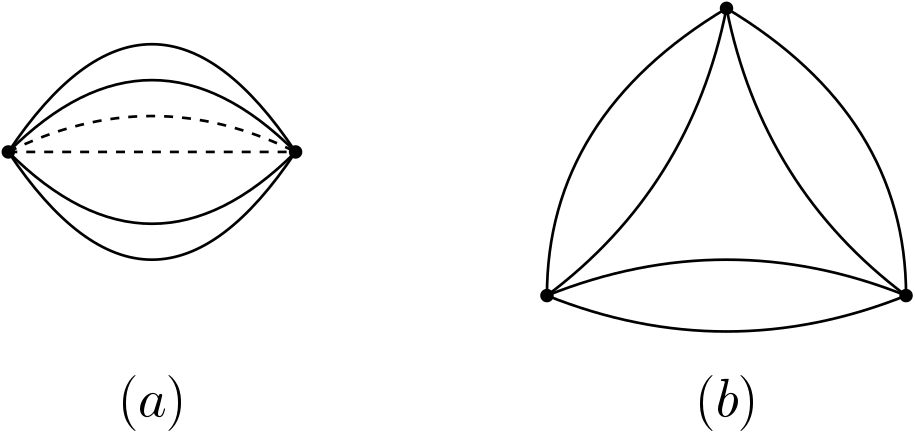}
    \caption{$(a)$: Banana graph, $(b)$: Graph with $|V(G)| = 3$}
    \label{fig:banana}
\end{figure}

If $|V(G)| < 3$, we know that $G$ is either a single point or the path $P_2$ (both of which have gonality $1$), or that $G$ is a  banana graph on two vertices, which has gonality $2$ (see Figure \ref{fig:banana}(a)). Hence, if $|V(G)| \leq 3$ and $\text{gon}(G) = 3$, then $|V(G)| = 3$. Notice that, as in Figure \ref{fig:banana}(b), we can have a $4$-edge-connected graph $G$ with $\text{gon}(G) = 3$ and $|V(G)| = 3$.
\end{proof}

As an aid for the reader, we introduce the $3$-edge-connected graph depicted in Figure \ref{fig:runningex}. After proving each of the following lemmas, we will demonstrate the effect of the result on this graph, culminating in the construction of a non-degenerate harmonic morphism down to a tree.

\begin{figure}[ht]
    \centering
    \includegraphics[scale=0.3]{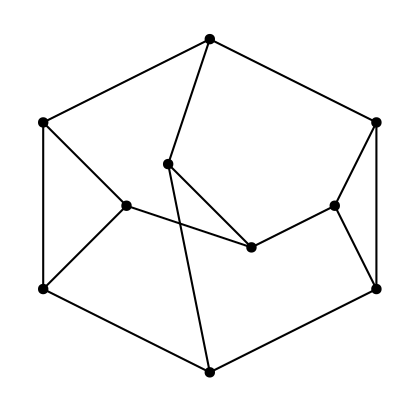}
    \caption{Running example graph used in Section \ref{section:multi}}
    \label{fig:runningex}
\end{figure}

For the next two lemmas, let $G$ be a simple, $3$-edge-connected graph with $\text{gon}(G)=3$, and let $D$ be a divisor on $G$ of rank $1$ and degree $3$.

\begin{lemma}
\label{lemma:distinctorbits}
For any vertex $v\in V(G)$, there is a unique divisor $D' \in \text{Div}_+(G)$ such that $D \sim D'$ and $v \in \text{supp}(D')$.
\end{lemma}

\begin{proof}
Since $r(D)=1$, we know that for any vertex $v_1\in V(G)$, there exist (not necessarily distinct) vertices $v_2, v_3 \in V(G)$ such that $D \sim (v_1) + (v_2) + (v_3)$. Thus for any $v\in V(G)$, there exists at least one divisor $D' \in \text{Div}_+(G)$ such that $D \sim D'$ and $v \in \text{supp}(D')$.

For uniqueness of $D'$, consider the \textit{Abel-Jacobi map} $S^{(k)} : \text{Div}_+^k(G) \to \text{Jac}(G)$ with basepoint $v_0$, defined as follows:
\[
S^{(k)}((v_1) + \cdots + (v_k)) = [(v_1) - (v_0)] + \cdots + [(v_k) - (v_0)],
\]
where $[(v)]$ denotes the equivalence class associated to the divisor $(v)$ under the usual equivalence relation on divisors. Then, by Theorem 1.8 from \cite{bn}, $S^{(k)}$ is injective if and only if $G$ is $(k+1)$-edge-connected.  Suppose now that $D \sim (v_1) + (v_2) + (v_3)$ and that there exist two other vertices $v_2', v_3'$ satisfying $D \sim (v_1) + (v_2') + (v_3')$. Then, we see that
\[
(v_2) + (v_3) - 2(v_1) \sim D - 3(v_1) \sim (v_2') + (v_3') - 2(v_1).
\]
Since $G$ is $3$-edge-connected, the Abel-Jacobi map with basepoint $v_1$ is injective, and so up to relabelling we have $v_2 = v_2'$ and $v_3 = v_3'$.  Thus $D'$ is unique.
\end{proof}

Note that a generalization of Lemma \ref{lemma:distinctorbits} for divisors of degree $d$ and rank $1$ on $d$-edge-connected graphs for arbitrary $d \in \mathbb{Z}_{>0}$ is proven in \cite{cp}. 

For a vertex $v \in V(G)$, let $D_v$ denote the unique effective divisor satisfying both $D_v \sim D$ and $v \in \text{supp}(D_v)$. Define a new equivalence relation $\sim_D$ on $V(G)$ with $v_1 \sim_D v_2$ if and only if $v_1 \in \text{supp}(D_{v_2})$ and $v_2 \in \text{supp}(D_{v_1})$. The equivalence classes associated with this relation are
\[
[v]_D \coloneqq \{ v' \in V(G) : v' \in \text{supp}(D_v) \}.
\]

We define a morphism $\varphi: G \to G / \sim_D$ in the following way:

\begin{itemize}
    \item[(i)] If $v \in V(G)$, then $\varphi(v) = [v]_D$.
    \item[(ii)]  If $e = xy \in E(G)$ and $\varphi(x) = \varphi(y)$, then $\varphi(e) = [x]_D = [y]_D$.
    \item[(iii)]  If $e = xy \in E(G)$ and $\varphi(x) \neq \varphi(y)$, then $\varphi(e) = [e]_D$ (where $[e]_D$ has endpoints $[x]_D$ and $[y]_D$).
\end{itemize}

In our running example from Figure \ref{fig:runningex}, define $D$ to be the divisor with one chip on every vertex in the left-most $3$-cycle of the graph. Figure \ref{fig:partitioning} shows the partitioning of vertices into equivalence classes on our running example graph, as well as the effect of $\varphi$ on the graph.

\begin{figure}[ht]
    \centering
    \includegraphics[scale=0.4]{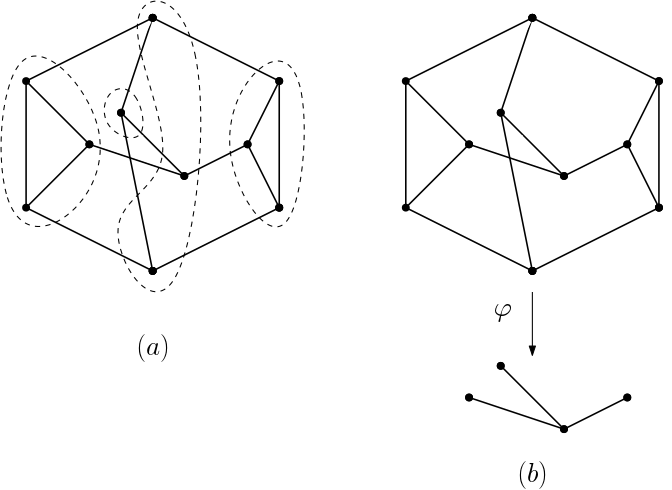}
    \caption{$(a)$ Vertex partition, $(b)$ Quotient morphism $\varphi$}
    \label{fig:partitioning}
\end{figure}

Since we will use this quotient morphism in our proof of Theorem \ref{thm:multigraph}, we now prove the following lemma, which will aid us in showing that this morphism is harmonic.

\begin{lemma}
\label{lemma:threeedges}
If $e = uv \in E(G)$ such that $[u]_D \neq [v]_D$, then there exist exactly three edges between the vertices in $[u]_D$ and the vertices in $[v]_D$. Furthermore, for any given vertex $u \in [u]_D$, 
\[
| E(u, [v]_D) | = D_u(u).
\]
\end{lemma}

\begin{proof}
Suppose we have an edge $e = uv \in E(G)$ and we begin with the divisor $D_u$ satisfying $\text{supp}(D_u) = [u]_D$. By Lemmas \ref{lemma:finiteseq} and \ref{lemma:uniquevred}, there exists a unique $v$-reduced divisor equivalent to $D_u$ which can be reached by a finite sequence of chip-firing moves. Furthermore, by Lemma 3.21 of \cite{db}, we never need to fire from the vertex $v$ itself during the reduction process. Since $u$ and $v$ are connected by an edge, our first chip-firing move must move at least one chip onto $v$; otherwise, we would have fired a collection of vertices not including $u$, thereby obtaining another effective divisor $D'$ with $u\in\text{supp}(D')$, which contradicts the uniqueness of $D_u$. However, by the uniqueness of the divisor $D_v$ with $v \in \text{supp}(D_v)$, we must have moved all three chips onto $[v]_D$ with this single chip-firing move. This implies that there exist at least three edges from $[u]_D$ to $[v]_D$ because only one chip can be sent along any given edge. On the other hand, because we were able to successfully fire our three chips from $[u]_D$ onto $[v]_D$ without inducing debt, this also implies that there are at most three edges and that the number of edges outgoing from each vertex $u \in [u]_D$ is equivalent to $D_u(u)$. This establishes our claim.
\end{proof}

Note that in our running example, the partition of vertices depicted in Figure \ref{fig:partitioning}$(a)$ shows that there are exactly three edges between every pair of adjacent vertex classes, and that the number of edges incident with each vertex $v$ in the class is precisely the number of chips on $v$ in the associated divisor $D_v$. 

Armed with these results, we can now prove the main result of this section.

\begin{proof}[Proof of Theorem \ref{thm:multigraph}]
We will first show that $(1) \implies (2)$. Let $G$ be a graph of gonality $3$. If $|V(G)| \leq 3$, then by the proof of Lemma \ref{prop:allgon3}, we know that $\text{gon}(G) = 3$ only if $|V(G)| = 3$. 

Assume now that $|V(G)| > 3$. Since $\text{gon}(G)=3$, there exists a divisor $D \in \text{Div}_+(G)$ such that $\text{deg}(D) = 3$ and $r(D) = 1$. Define the equivalence relation $\sim_D$ as before, with $[v]_D$ again referring to the equivalence class associated to $v$ under $\sim_D$.  Let $\varphi$ be the quotient morphism $\varphi : G \to G / \sim_D$ defined above. We will now show that $\varphi$ is a non-degenerate harmonic morphism of degree $3$.

By Lemma \ref{lemma:threeedges}, we have
\[
m_\varphi(v) = | \{ e \in E(G) : v \in e, \varphi(e) = [e]_D \} | = D_{v}(v),
\]
for each $[e]_D \in E(G/\sim_D)$ such that $[v]_D \in [e]_D$. The assumption that $|V(G)| > 3$ ensures that we have at least one edge between vertices in different equivalence classes. Since $D_{v}(v) > 0$ for each $v \in [v]_D$, our morphism is non-degenerate. Furthermore, since $m_\varphi(v) = D_v(v)$ does not depend on our choice of $[e]_D \in E(G/\sim_D)$, our morphism is harmonic. Hence, 
\[
\text{deg}(\varphi) = \sum_{v \in [v]_D} |\{ e \in E(G) : v \in e, \varphi(e) = [e]_D \}| = \sum_{v \in [v]_D} D_{v}(v) = 3.
\]

We will now show that $\varphi(G) = G / \sim_D$ is a tree.  We define the \textit{pullback} map $\varphi^* : \text{Div}(G') \to \text{Div}(G)$ associated to a harmonic morphism $\varphi : G \to G'$ as
\[
(\varphi^*(D'))(v) = m_{\varphi}(v) \cdot D'(\varphi(v)).
\]
For any two vertices $x, y \in \varphi(G)$, we see that

\begin{align*}
    \varphi^*((x)) &= \sum_{\substack{v \in V(G), \\ \varphi(v) = x}} m_\varphi(v) \cdot (v) 
    = \sum_{v \in [x]_D} D_{v}(v) \cdot (v) 
    \\&\sim \sum_{v' \in [y]_D} D_{v'}(v') \cdot (v') 
    = \sum_{\substack{v' \in V(G), \\ \varphi(v') = y}} m_\varphi(v') \cdot (v') 
    = \varphi^*((y)).
\end{align*}
By Theorem 4.13 from \cite{bn2}, the induced homomorphism $\overline{\varphi} : \text{Jac}(G') \to \text{Jac}(G)$ is injective. Since $\varphi^*((x)) \sim \varphi^*((y))$, we find that $(x) \sim (y)$. Applying Lemma 1.1 of \cite{bn2} now shows that $G / \sim_D$ is a tree.

For the reverse direction $(2) \implies (1)$, first suppose that $|V(G)| = 3$. Then, by Lemma \ref{lemma:eta},  $$\text{gon}(G) \geq \min \{ \eta(G), |V(G)| \} = 3$$ and by Lemma \ref{lemma:upperbound}, $\text{gon}(G) \leq 3$.  Suppose now that there exists a non-degenerate harmonic morphism $\varphi : G \to T$ such that $\text{deg}(\varphi) = 3$ and $T$ is a tree. Fix $x_0 \in T$ and let 
 \[
 D = \varphi^*((x_0)) = \sum_{v \in V(G) : \varphi(v) = x_0} m_\varphi(v) \cdot (v).
 \]
 It is clear that $D$ is effective and by Lemma 2.13 in \cite{bn2}, $\text{deg}(D) = 3$. We claim that $r(D) \geq 1$. Pick $x \in G$. Since $T$ is a tree, by Lemma 1.1 from \cite{bn2}, $(\varphi(x)) \sim (x_0)$. Now, by Proposition 4.2 (again from \cite{bn2}), 

 \begin{align*}
 D = \varphi^*((x_0)) \sim \varphi^*((\varphi(x))) &= \sum_{v \in V(G) : \varphi(v) = \varphi(x)} m_\varphi(v) \cdot (v) \\
 &= m_\varphi(x) \cdot (x) + \sum_{v \in V(G) : v \neq x, \varphi(v) = \varphi(x)} m_\varphi(v) \cdot (v) \\
 &= m_\varphi(x) \cdot (x) + E,
 \end{align*}
 where $E$ is an effective divisor. Notice that because $\varphi$ is non-degenerate, $m_\varphi(x) > 0$ so $D \sim c \cdot (x) + E$ for each $x \in V(G)$ where $c \in \mathbb{Z}_{>0}$. Hence, we find that $r(D) \geq 1$. By Lemma \ref{lemma:eta}, it follows that $\text{gon}(G) = 3$.
\end{proof}

Theorem \ref{thm:multigraph} can be applied to determine the geometric gonalities of graphs with known divisorial gonalities. For example, consider the $3$-cube graph $Q_3$ illustrated in Figure \ref{fig:cube}, which is $3$-edge-connected. It can be computationally verified that $\text{gon}(Q_3) = 4$. Since this graph is not a tree and doesn't have divisorial gonality $2$ or $3$, we know by Theorems \ref{thm:chan} and \ref{thm:multigraph} that $\text{ggon}(Q_3) \geq 4$. Figure \ref{fig:cube} depicts a non-degenerate harmonic morphism on $Q_3$ of degree $4$, so we must have $\text{ggon}(Q_3) = 4$.

\begin{figure}[hbt]
    \centering
    \includegraphics[scale=0.3]{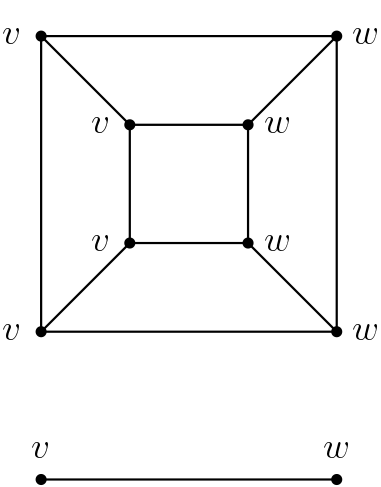}
    \caption{The $3$-cube $Q_3$ with gonality $4$}
    \label{fig:cube}
\end{figure}

We can also apply Theorem~\ref{thm:multigraph} to certain graphs with bridges, assuming that they become $3$-edge-connected after contracting these bridges. This is due to the following proposition, which comes as an immediate consequence of Corollary 5.10 in \cite{bn2} on rank-preservation under bridge contraction: 

\begin{proposition}
\label{prop:bridgecontract}
If $G$ is a graph and $G'$ is the graph obtained by contracting every bridge of $G$, then $\text{gon}(G) = 3$ if and only if $\text{gon}(G') = 3$.
\end{proposition}

\section{Simple Graphs of Gonality Three}
\label{section:simple}

We now restrict our attention to graphs that are simple.  We say a graph of gonality $3$ \emph{satisfies the zero-three condition} if there exists a divisor $D$ such that $\deg(D)=3$, $r(D)=1$, and for any three distinct vertices $a,b,c$ with $D\sim (a)+(b)+(c)$, we have that $a,b,c$ either share $0$ edges or share $3$ edges.

The following theorem extends Theorem \ref{thm:multigraph} by adding an extra equivalent statement, under a few stronger assumptions.

\begin{theorem}
\label{thm:simpletheorem}
If $G$ is a simple, $3$-vertex-connected combinatorial graph satisfying the zero-three condition, then the following are equivalent:
\begin{enumerate}
    \item $G$ has gonality $3$.
    \item There exists a non-degenerate harmonic morphism $\varphi : G \to T$, where $\text{deg}(\varphi) = 3$ and $T$ is a tree.
    \item There exists a cyclic automorphism $\sigma: G \to G$ of order $3$ that does not fix any edge of $G$, such that $G/\sigma$ is a tree. 
\end{enumerate}
The equivalence of (1) and (2) and the implication (3) implies (1) and (2) still hold without the zero-three condition assumption.
\end{theorem}

Notice that we no longer need to worry about the case where $|V(G)| = 3$; this is because there are no simple $3$-vertex-connected graphs with exactly $3$ vertices. Also note that while statements (1) and (2) in Theorem \ref{thm:simpletheorem} are nearly identical to those given in Theorem \ref{thm:chan}, statement (3) now requires the extra condition that the automorphism $\sigma$ does not fix any edge of $G$. In our proof of this theorem, we will show that this condition is required for the implication $(3) \implies (2)$ to hold.

\begin{figure}[hbt]
\centering
\includegraphics[scale=1]{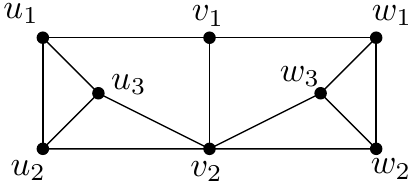}
\caption{A $3$-edge-connected graph of gonality $3$ without an automorphism of order $3$}
\label{fig:no_automorphism}
\end{figure}

We should also remark why we need the stronger assumptions that our graph is $3$-vertex-connected instead of $3$-edge-connected, and that our graph satisfies the zero-three condition.  Consider the graph $G$ in Figure \ref{fig:no_automorphism}.  The divisor $(v_1)+2(v_2)$ has positive rank, so $G$ has gonality at most $3$; and since $G$ has $K_4$ as a minor, the treewidth of the graph, and thus its gonality, is at least $3$ by \cite{bod} and Lemma \ref{ref:treewidth}.  Moreover, $G$ is $3$-edge-connected, although it is not $3$-vertex-connected, since removing $v_1$ and $v_2$ disconnects the graph.  (The graph $G$ does happen to satisfy the zero-three condition, via the divisor $(v_1)+2(v_2)$.) Finally, let us determine the automorphism group of $G$.  Any automorphism must send $v_2$ to $v_2$ since it is the only vertex of degree $5$, and $v_1$ to $v_1$ since it is the only vertex not on a cycle of length $3$.  From there, using adjacency relations of vertices we can determine that $\textrm{Aut}(G)$ is isomorphic to $\mathbb{Z}/2\mathbb{Z}\times \mathbb{Z}/2\mathbb{Z}\times \mathbb{Z}/2\mathbb{Z}$, and can be generated by three automorphisms of order $2$: one switching $u_2$ and $u_3$, one switching $w_2$ and $w_3$, and one switching $u_i$ with $w_i$ for all $i$ in $\{1,2,3\}$.  Since $|\textrm{Aut}(G)|=8$, the graph $G$ does not have an automorphism of order $3$, even though it is $3$-edge-connected and has gonality $3$.

\begin{figure}[hbt]
\centering
\includegraphics[scale=1]{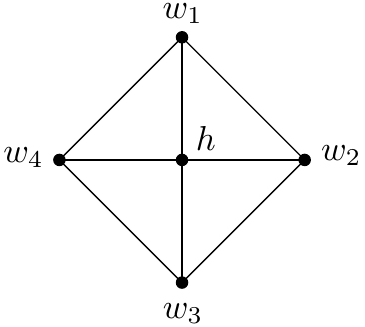}
\caption{The wheel graph $W_5$}
\label{fig:wheel}
\end{figure}

To see that the zero-three condition is necessary, consider the graph in Figure \ref{fig:wheel}, which is the wheel graph $W_5$ on $5$ vertices.  It has gonality $3$:  the divisor $(w_1)+(h)+(w_3)$ has positive rank; and since the graph has $K_4$ as a minor, the treewidth of the graph, and thus its gonality, is at least $3$ by \cite{bod} and Lemma \ref{ref:treewidth}.  It is also $3$-vertex-connected.  However, we claim that it does not satisfy the zero-three condition.  Let $D$ be any effective rank $1$ divisor of degree $3$ on  $W_5$.  We will see in Lemma \ref{lemma:1or3} that $D$ must have support size $1$ or $3$, and that it is equivalent to some divisor with support size $3$.  If $W_5$ satisfies the zero-three condition, then since any three vertices have at least one edge in common, there must be a rank $1$ divisor $(a)+(b)+(c)$ on $W_5$ where $a,b,c$ are distinct and form a $K_3$ in the graph.  It follows that $(w_i)+(w_{i+1})+(h)$ has rank $1$ for some $i$, where addition is done modulo $4$.  However, this divisor does not have rank $1$:  starting Dhar's burning algorithm from $w_{i+2}$ burns the whole graph.  Thus $W_5$ must not satisfy the zero-three condition.  Finally, the automorphism group of $W_5$ is the same as the automorphism group of the cycle $C_4$, namely the dihedral group of order $8$.  This group does not have any elements of order $3$.

Now, let $D$ be a divisor of degree 3 and rank $1$ on a graph of gonality $3$.  Recall the equivalence relation $\sim_D$ on $V(G)$ resulting from Lemma \ref{lemma:distinctorbits}: $v_1 \sim_D v_2$ if and only if $v_1 \in \text{supp}(D_{v_2})$ and $v_2 \in \text{supp}(D_{v_1})$.  We will use this relation to define a permutation $\sigma$ of the vertices of $G$, which we will then prove to be a cyclic automorphism of order $3$.  First, we need the following lemma.

\begin{lemma}
\label{lemma:1or3}  Let $G$ be a simple, $3$-vertex-connected graph with $\text{gon}(G) = 3$, with $D \in \text{Div}_+(G)$ such that $r(D) = 1$ and $\text{deg}(D) = 3$.
If $D\sim (v_1)+(v_2)+(v_3)$, then either $v_1=v_2=v_3$ or $v_1$, $v_2$, and  $v_3$ are all distinct.
\end{lemma}

\begin{proof}
Suppose for the sake of contradiction that there exists a divisor $D' \in \text{Div}_+(G)$ such that $D' \sim D$ and $D' = 2(v_1) + (v_2)$ where $v_1 \neq v_2$. Let $v_0\in V(G)$ be distinct from $v_1$ and $v_2$, and start Dhar's burning algorithm at $v_0$.  Since $G$ is $3$-vertex-connected, the graph $G-\{v_1,v_2\}$ is connected, so every vertex in $G$ besides $v_1$ and $v_2$ will burn.  Since $\deg(v_2)\geq 3$ and since $G$ is simple, there are at least two edges connecting $v_2$ to $G-\{v_1,v_2\}$.  Both these edges are burning, so $v_2$ burns since it only has one chip.  At this point the whole graph besides $v_1$ is burning.  Since $\deg(v_1)\geq 3$, there are at least three burning edges incident to $v_1$.  Since $v_1$ has two chips, $v_1$ burns, and thus the entire graph burns.  This shows that $D'$ is $v_0$-reduced. Since $v_0 \notin \text{supp}(D')$, $D'$ is not a winning divisor, which is a contradiction to $D' \sim D$.
\end{proof}

\begin{figure}[hbt]
\centering
\includegraphics[scale=0.3]{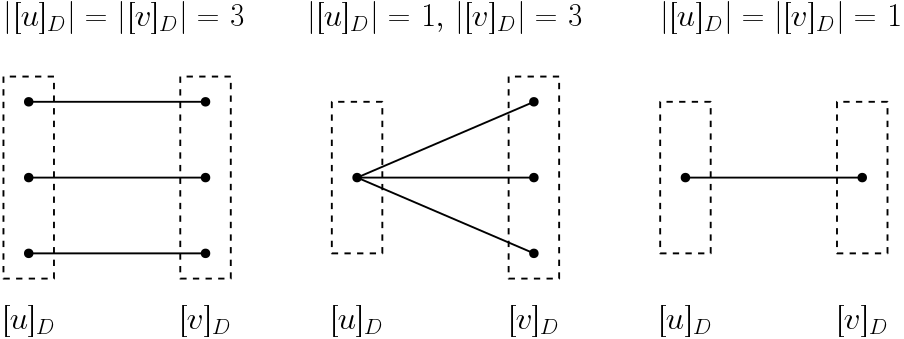}
\caption{Three possible edge cases for Lemma \ref{lemma:threeedges}}
\label{fig:edgecases}
\end{figure}

Now, we again consider the equivalence classes $[v]_D$ of $V(G)$ under the relation $\sim_D$.  Recall from Lemma \ref{lemma:threeedges} that if $uv\in E(G)$ and $[u]_D\neq [v]_D$, then there are exactly three edges between the vertices in $[u]_D$ and in $[v]_D$.  By Lemma \ref{lemma:1or3}, we have $|[v]_D| = 1$ or $|[v]_D| = 3$ for each vertex $v \in V(G)$. The cases depicted in Figure \ref{fig:edgecases} are thus the only possible edge configurations between equivalence classes. Furthermore, note that the third case in Figure \ref{fig:edgecases} is impossible because we cannot send all three chips along a single edge at once.

Define $\sigma : V(G) \to V(G)$ to be the map which permutes the vertices of each equivalence class $[v]_D$ in the following fashion: Choose a vertex $v_1 \in V(G)$ such that $|\text{supp}(D_{v_1})| = 3$. Such a vertex must exist because otherwise, since our graph is connected and not a single vertex, we would be in the third case of Figure \ref{fig:edgecases}. For the unique vertices $v_1, v_2, v_3 \in [v_1]_D$, let $\sigma(v_1) = v_2$, $\sigma(v_2) = v_3$, and $\sigma(v_3) = v_1$. For each edge from a vertex in $[v_1]_D$ to a vertex in another equivalence class $[u_1]_D$, define $\sigma$ as follows. If $e = v_1u_1 \in E(G)$ with $|\text{supp}(D_{u_1})| = 3$, then let $\sigma(u_1) = u_2$ where $u_2 \in [u_1]_D$ is the unique vertex such that $\sigma(v_1)u_2 \in E(G)$. Then we must have $\sigma(u_2) = u_3$ where $u_3$ is the unique vertex with $\sigma(v_2)u_3 \in E(G)$. On the other hand, if $|\text{supp}(D_{u_1})| = 1$, then let $\sigma$  act as the identity on $u_1$. 

Let this process, where vertex classes induce orderings on their adjacent vertex classes, propagate outwards. If we reach a situation where a vertex class with one vertex induces an order on a vertex class with three vertices, pick some arbitrary ordering on those three vertices and define $\sigma$ accordingly. We will show that the order chosen does not matter, and that this process provides us with our desired automorphism.

\begin{prop}
\label{prop:cyclic}
If $G$ satisfies the zero-three condition, then the map $\sigma$ is a cyclic automorphism of order $3$ that does not fix any edge of $G$.
\end{prop}

\begin{proof}
First consider the action of $\sigma$ on a triple $v_1,v_2,v_3$ where $D\sim(v_1)+(v_2)+(v_3)$ and $v_1,v_2,v_3$ are all distinct.  We are assuming our graph satisfies the zero-three condition, so the map $\sigma$ mapping $v_1$ to $v_2$ to $v_3$ to $v_1$ preserves the connectivity of $v_1,v_2,$ and $v_3$, since either all share edges or none share edges.

Since our graph $G$ is connected, the propagation process induces an order on each vertex class in $G$. We now argue that we never run into the problem that the induced orderings are incompatible with each other. Suppose for the sake of contradiction that we have a vertex class $[v]_D$ with one ordering induced by an adjacent class $[u]_D$ and another ordering induced by an adjacent class $[w]_D$. It is clear that $[u]_D \neq [w]_D$ and that $| [v]_D | = 3$. However, this implies that there are at least two paths from each vertex in our original vertex class $[v_1]_D$ into $[v]_D$. Consider the divisor $D_{v}$ with $[v]_D \in \text{supp}(D_v)$, and begin Dhar's burning algorithm at our starting vertex from $[v_1]_D$.  We will show that no matter how the algorithm runs, we reach a contradiction, either to $r(D)\geq 1$ or to $[u]_D\neq [w]_D$.

Let $v_a,v_b,v_c$ be the three vertices in $[v]_D$. First, assume that at least one vertex in $[v]_D$ burns, say $v_a$.  Since $G$ is $3$-vertex-connected, it is still connected after removing $v_b$ and $v_c$, so every other vertex in $G$ must burn.  We also know that $\deg(v_b)$ and $\deg(v_c)$ are both at least $3$, and so these vertices are adjacent to at least two burning vertices.  Since each has one chip, both of these vertices (and thus the entire graph) will burn. This is a contradiction, since $r(D)\geq 1$.

Now assume none of $v_a,v_b,$ and $v_c$ burns. Since there are two burning edges coming into $[v]_D$, one from $[u]_D$ and one from $[w]_D$, these burning edges must be incident to different vertices among $v_a$, $v_b$, and $v_c$; otherwise one of these vertices would burn. When the whole graph does not burn, Dhar's burning algorithm fires chips from all unburnt vertices, which moves a chip to $[u]_D$ and a chip to $[w]_D$. This yields a rank $1$ divisor of degree $3$ with support in both $[u]_D$ and $[w]_D$, a contradiction to $[u]_D\neq [w]_D$ by Lemma \ref{lemma:distinctorbits}.

Having reached a contradiction in all cases, we know that the propagation process can never lead to incompatible orderings. Notice also that if $e = uv \in E(G)$, then $\sigma^{-1}(u)\sigma^{-1}(v) \in E(G)$ because $\sigma^{-1}(u) = \sigma^2(u)$ and $\sigma^{-1}(v) = \sigma^2(v)$. Hence, we have shown that $\sigma$ is an automorphism. By definition, $\sigma$ is cyclic and we have already demonstrated that $\sigma$ has order $3$. Finally, we see that $\sigma$ does not fix any edge of $G$ because we have already shown that we cannot have an edge between two equivalence classes with one vertex each (recall that the third edge case in Figure \ref{fig:edgecases} is impossible).
\end{proof}

We may now define the same quotient morphism $\varphi: G \to G / \sim_D$ as in Section \ref{section:multi}. However, notice that our equivalence classes of $V(G)$ can now be viewed as orbits under the action of $\sigma$. Thus, $G / \sim_D = G / \sigma$.

\begin{prop}
\label{prop:quotient}
The quotient morphism $\varphi: G \to G / \sigma$ is harmonic and nondegenerate. Moreover, $G/\sigma$ is a tree.
\end{prop}

The proof of this proposition will be similar to an argument from  the proof of Theorem \ref{thm:multigraph}, when we showed that the quotient map from $G$ to $G/\sim_D$ was harmonic of degree $3$, and that $G/\sim_D$ was a tree.

\begin{proof} Since $\sigma$ is an automorphism, for each vertex $v \in V(G)$ such that $v \in [v]_D$, there exists some edge $e \in E(G)$ such that $v \in e$ and $\varphi(e) = [e]_D$. (Otherwise $G/\sigma$ would not be connected.) Hence, $\varphi$ is non-degenerate. To show that $\varphi$ is harmonic, fix a vertex $v \in V(G)$ and consider all edges $[e]_D \in E(G/\sigma)$ such that $\varphi(v) = [v]_D \in [e]_D$. If $\left| [v]_D \right| = 3$, then $| \{e \in E(G) : v \in e, \varphi(e) = [e]_D \}| = D_v(v) = 1$, no matter which edge $e$ we pick. On the other hand, if $\left| [v]_D \right| = 1$, then $| \{e \in E(G) : v \in e, \varphi(e) = [e]_D \}| = D_v(v) = 3$, which is also independent of our choice of $e$. Hence, $\varphi$ is harmonic. 

For any given edge $[e]_D \in E(G/\sigma)$, 
\[
|\varphi^{-1}([e]_D)| = \sum_{v \in [v]_D} |\{ e \in E(G) : v \in e, \varphi(e) = [e]_D \}| = \sum_{v \in [v]_D} D_v(v),
\]
for any choice of $[v]_D$ such that $[v]_D \in [e]_D$. Thus, $\varphi$ is a degree $3$ morphism. The same argument from the proof of Theorem \ref{thm:multigraph} shows that $G / \sigma$ is a tree. 
\end{proof}

We are now ready to prove Theorem \ref{thm:simpletheorem}.

\begin{proof}[Proof of Theorem \ref{thm:simpletheorem}]
We already have the equivalence of (1) and (2) from Theorem \ref{thm:multigraph}.  
The implication $(1) \implies (3)$ follows from Proposition \ref{prop:cyclic}.

For $(3) \implies (2)$ (for which we will not assume $G$ satisfies the zero-three condition),  suppose that there exists a cyclic automorphism $\sigma:G\rightarrow G$ of order $3$ that does not fix any edge of $G$, such that $G/\sigma$ is a tree.  We wish to show that  $\varphi : G \to G/\sigma$ is a non-degenerate harmonic morphism of degree $3$.  The argument for this is nearly identical to the argument from Proposition~\ref{prop:quotient}.  However, the proof of harmonicity requires a few additional details. 

Fix a vertex $v \in V(G)$ and consider all edges $[e] \in E(G/\sigma)$ such that $\varphi(v) = [v] \in [e]$. Since $\sigma$ has order $3$, either $|\varphi^{-1}([v])|=3$ or $|\varphi^{-1}([v])|=1$. If $|\varphi^{-1}([v])|=3$, then we claim that $| \{e \in E(G) : v \in e, \varphi(e) = [e] \}| = 1$, no matter which edge $e$ we pick.  To see this, suppose there are multiple edges $e_1=vw_1$ and $e_2=vw_2$ in this set. Then, without loss of generality, $\sigma(e_1) = e_2$ and since $v$ is not fixed under $\sigma$, we must have $\sigma(v) = w_2$. But then $\varphi$ sends $v$ and $w_2$ to the same point, so $e_2$ is mapped to a point rather than an edge, a contradiction. On the other hand, if $|\varphi^{-1}([v])|=1$, then we claim that $| \{e \in E(G) : v \in e, \varphi(e) = [e] \}| = 3$, no matter which edge $e$ we pick. This is because $\sigma$ fixes $v$ but does not fix any of its incident edges. Hence, these edges must cycle around $v$ with order $3$. We have now shown that $\varphi$ is harmonic. By the same computation as in Proposition \ref{prop:quotient}, $\varphi$ also has degree $3$. \end{proof}

Unlike Theorem \ref{thm:multigraph}, which only relates divisorial and geometric gonalities, Theorem \ref{thm:simpletheorem} allows us to determine divisorial gonalities using information about graph automorphisms. For example, consider the Frucht graph in Figure \ref{fig:frucht}, which is the smallest trivalent graph with no nontrivial automorphisms \cite{frucht}.  It can be computationally verified that the divisor depicted in Figure \ref{fig:frucht} is indeed a winning divisor, so the Frucht graph has divisorial gonality at most $4$.  It is $3$-vertex-connected, and since it has no cyclic automorphisms of order $2$ or $3$, so either its gonality is $4$, or its gonality is $3$ and it does not satisfy the zero-three condition.  We computationally run through all possible support sets for divisors of degree $3$ with exactly $1$ or $2$ edges between the three vertices, and verify none have positive rank.  Thus the Frucht graph has gonality $4$.

\begin{figure}[ht]
    \centering
    \includegraphics[scale=0.25]{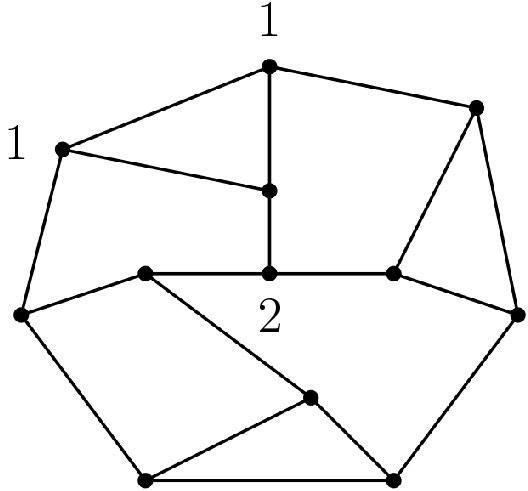}
    \caption{The Frucht graph with gonality $4$}
    \label{fig:frucht}
\end{figure}

Having proved both Theorem \ref{thm:multigraph} and Theorem \ref{thm:simpletheorem} we have established our main result, Theorem \ref{thm:maintheorem}.  We now ask whether it can be strengthened.  In particular, since the condition of being $3$-vertex-connected is relatively strong, we might wonder whether a weaker condition, such as being trivalent, is sufficient for Theorem \ref{thm:maintheorem} to hold. The next result shows that this is not the case.

\begin{prop}
\label{prop:twoconnect}
If $G$ is a simple, bridgeless trivalent graph that is not $3$-vertex-connected, then $\text{gon}(G) \geq 4$. 
\end{prop}

\begin{proof}
First we note that a trivalent graph is $3$-vertex-connected if and only if it is $3$-edge-connected \cite[Theorem 5.12]{cz},  so our graph $G$ is \textit{not} $3$-edge-connected.  Since $G$ is also bridgeless, it must be exactly $2$-edge-connected. This means that there exists some way to partition $G$ into two subgraphs, $H_1$ and $H_2$, connected by exactly two edges, as illustrated in Figure \ref{fig:twobridge}. Suppose for the sake of contradiction that there exists $D \in \text{Div}_+(G)$ with $\text{deg}(D) = 3$ and $r(D) = 1$. Then there exists some divisor $D' \sim D$ such that $D'$ has exactly two chips on $H_1$ and one chip on $H_2$: we must be able to move at least one chip onto both subgraphs, and since there are only two edges connecting the subgraphs, we can move at most two chips in a single subset firing move.

\begin{figure}[ht]
    \centering
    \includegraphics[scale=0.2]{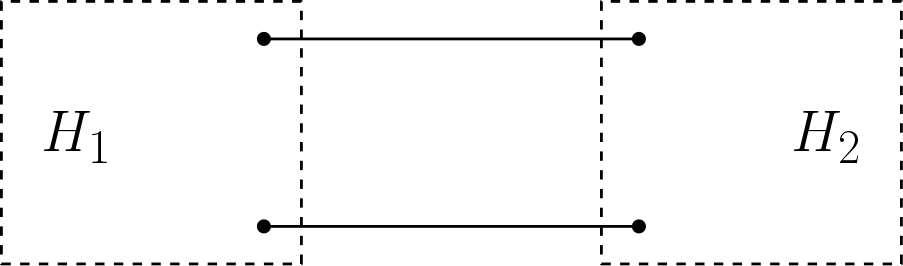}
    \caption{Simple, trivalent, exactly $2$-edge-connected graph}
    \label{fig:twobridge}
\end{figure}

Let $v_1, v_2 \in H_1$ and $v_3 \in H_2$ be the vertices such that $\text{supp}(D') = \{v_1,v_2,v_3\}$. We will split into two cases: first where removing $v_1$ and $v_2$ disconnects the graph, and second where removing them leaves the graph connected.  We split the former case into two subcases, depending on the relationship of $v_1$ and $v_2$ to a connected component of the disconnected graph which does not contain $v_3$.

Assume that removing $v_1$ and $v_2$ disconnects the graph into at least two connected components. A trivalent $2$-edge-connected graph is also $2$-vertex-connected, so it follows that $v_1 \neq v_2$. Let $H_3$ be one of the connected components which does \textit{not} contain $v_3$. This implies that there exists at least one edge incident to both $v_1$ and some vertex in $H_3$, and that the same holds for $v_2$. Since each vertex is trivalent, by symmetry, we have at most two edges connecting each vertex in $\{v_1,v_2\}$ with vertices in $H_3$. 

First we deal with the subcase that there exist at least two edges incident to either $v_1$ or $v_2$ entering $H_3$, and at least one edge incident to the other vertex entering $H_3$. Notice that we are at a state where we cannot fire onto $H_3$ without inducing debt (see the bottom graph in Figure \ref{fig:disconnect}). Choose a vertex $v_0\in H_3$. Since we are unable to fire without inducing debt, at least one of our two vertices has fewer chips than edges incident to $H_3$. Hence, if we begin Dhar's burning algorithm at $v_0$, everything in $H_3$ must burn, including at least one of the two vertices with chips. This forces the other vertex with a chip to burn as well. Since we have only one other vertex with exactly one chip, this implies that the whole graph burns.

Now we handle the subcase that there exist exactly one edge incident to $v_1$ and exactly one edge incident to $v_2$ entering $H_3$. Then there exist two vertices $v_1', v_2'\in H_3$ which are the endpoints of these edges (see the top graph in Figure \ref{fig:disconnect}). We know that $v_1'\neq v_2'$ because $G$ is $2$-vertex-connected. Fire onto $H_3$, moving the two chips from $\{v_1, v_2\}$ onto $\{v_1',v_2'\}$. Suppose that we can continue firing in this manner, \ie moving chips onto two vertices which are each connected by exactly one edge to the rest of the graph. Since our graph is finite, this process must terminate at some point. If we are able to hit all vertices in $H_3$, we have a contradiction because this implies that at least two vertices in $H_3$ are \textit{not} trivalent. Before hitting all of the vertices in $H_3$, we reach a state as in the previous case, with at least two edges incident to either of the two vertices entering the subgraph of $H_3$ that we are unable to fire onto. (Notice that we cannot fire from either vertex separately either, because this would imply the existence of a bridge.)

We initially assumed that $r(D) = 1$, so we have reached a contradiction. Thus, we know that removing the set $\{v_1, v_2\}$ cannot disconnect the graph.

\begin{figure}[ht]
    \centering
    \includegraphics[scale=0.3]{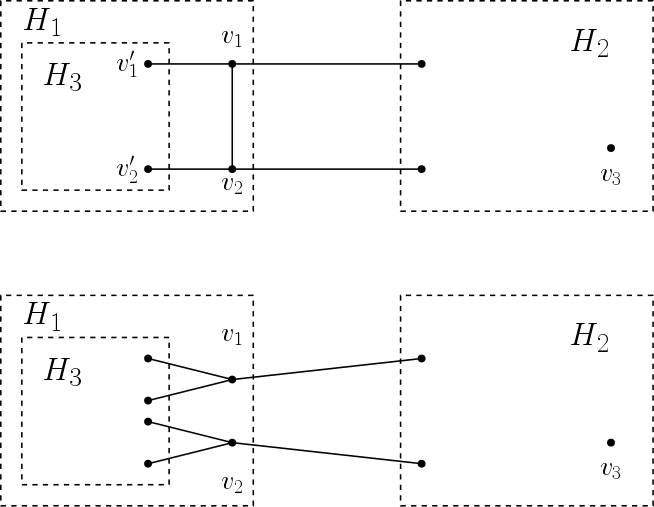}
    \caption{Two cases: removal of $\{v_1,v_2\}$ disconnects graph}
    \label{fig:disconnect}
\end{figure}

Now we assume that removing $v_1$ and $v_2$ does not disconnect the graph. Choose a vertex $v_0' \in H_2$ such that $v_0' \neq v_3$ (such a vertex exists due to trivalence). If we begin Dhar's burning algorithm at $v_0'$, we find that the entirety of $H_2$ must burn, since there exists only one vertex with a single chip in $H_2$. The fire then spreads across the two edges incident to $H_1$. Since removing $v_1$ and $v_2$ does not disconnect the graph, the fire must burn every vertex in $H_1$ except possibly $v_1$ and $v_2$. However, because our graph is simple, there exists at most one edge between $v_1$ and $v_2$, implying that each must have at least two incident burning edges. Hence, the whole graph burns, implying that $r(D) < 1$. Again, this is a contradiction.  We conclude that the gonality of the graph is at least $4$.
\end{proof}

It is worth noting that this result does not extend to multigraphs. Figure \ref{fig:gon3multi} depicts an example of a graph which is bridgeless, trivalent, and not $3$-vertex-connected, but has gonality $3$.

\begin{figure}[ht]
    \centering
    \includegraphics[scale=0.15]{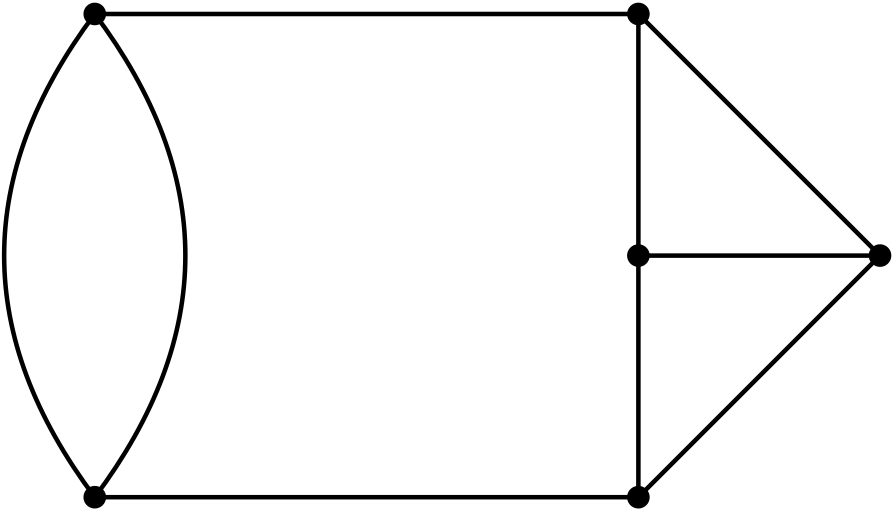}
    \caption{Multigraph $G$ with $\text{gon}(G) = 3$}
    \label{fig:gon3multi}
\end{figure}

\begin{coro}
If $G$ is a simple, bridgeless trivalent graph that is not $3$-vertex-connected, then $\text{ggon}(G) \neq 3$.
\end{coro}

\begin{proof}
By Proposition \ref{prop:twoconnect}, $G$ does not have gonality $3$. Notice that arguing that $(2) \implies (1)$ in the proof of Theorem \ref{thm:multigraph} does not require $3$-edge-connectivity. Hence, there exists no non-degenerate harmonic morphism of degree $3$ from $G$ to a tree.
\end{proof}

\section{Constructing Graphs of Gonality 3}
\label{section:constructions}

In \cite{chan}, Chan presents the following construction for all trivalent, $2$-edge-connected graphs of gonality $2$. Choose a tree $T$ where each vertex $v \in V(T)$ satisfies $\text{val}(v) \leq 3$.
\begin{enumerate}
    \item Duplicate $T$, making two copies $T_1$ and $T_2$.
    \item For each vertex $v_1\in T_1$ with $\text{val}(v_1)\leq 2$, connect it to the matching vertex in $T_2$ with $3-\text{val}(v_1)$ edges.
\end{enumerate}

Every graph constructed in this way is called a \emph{ladder}. By \cite[Theorem 4.9]{chan}, each graph arising from this construction has gonality $2$, and every $2$-edge-connected trivalent graph of gonality $2$ with genus at least $3$ comes from such a construction.

We now provide a similar construction for graphs of gonality $3$. In constrast to the results of \cite{chan}, not every graph of gonality $3$ arises from this construction.  For instance, the complete graph on $4$ vertices $K_4$ does not arise from this construction, even though it is $3$-vertex-connected and simple with gonality $3$; to see this, note that the number of vertices from our construction is always a multiple of $3$.

\begin{proposition}
\label{prop:construction}
Let $T$ be an arbitrary tree that is not a single vertex, and let $S\subset V(T)$ consist of at least two vertices. Construct a graph $\mathscr{T}(T)$ as follows:

\begin{enumerate}
    \item Duplicate $T$ twice, for a total of three copies of $T$. Call these copies $T_1$, $T_2$, and $T_3$.
    \item For each vertex $v_1 \in T_1$ with $v_1\in S$ and its corresponding vertices $v_2 \in T_2$ and $v_3 \in T_3$, connect each pair of vertices with an edge so that all three vertices are connected in a $3$-cycle.
\end{enumerate}
Then $\text{gon}(\mathscr{T}(T))=3$.
\end{proposition}

In the following proof, we refer to the \emph{Cartesian product} $G\boxempty H$ of two graphs $G$ and $H$.  This is the graph with vertex set $V(G)\times V(H)$, and an edge between $(u,u')$ and $(v,v')$ if and only if $u'=v'$ and $uv\in E(G)$, or $u=v$ and $u'v'\in E(H)$.

\begin{proof}
It is clear that the morphism $\varphi : \mathscr{T}(T) \to T$ which maps corresponding triples of vertices $\{v_1,v_2,v_3\}$ to each other is a non-degenerate harmonic morphism. Notice that arguing that $(2) \implies (1)$ in the proof of Theorem \ref{thm:multigraph} does not require $3$-edge-connectivity. Hence, there exists a divisor $D$ on $\mathscr{T}(T)$ such that $\text{deg}(D) = 3$ and $r(D) \geq 1$, meaning $\textrm{gon}(\mathscr{T}(T))\leq 3$.

Since $S$ contains at least two vertices, the graph $\mathscr{T}(T)$ has $K_2\boxempty K_3$ as a subgraph.  This graph in turn has $K_4$ as a minor, which is the forbidden minor of graphs of treewidth $2$ \cite{bod}.  Thus, $\textrm{gon}(\mathscr{T}(T))\geq \textrm{tw}(\mathscr{T}(T))\geq 3$ by Lemma \ref{ref:treewidth}.  We conclude that $\textrm{gon}(\mathscr{T}(T))=3$.
\end{proof}

\begin{figure}[ht]
    \centering
    \includegraphics[scale=0.35]{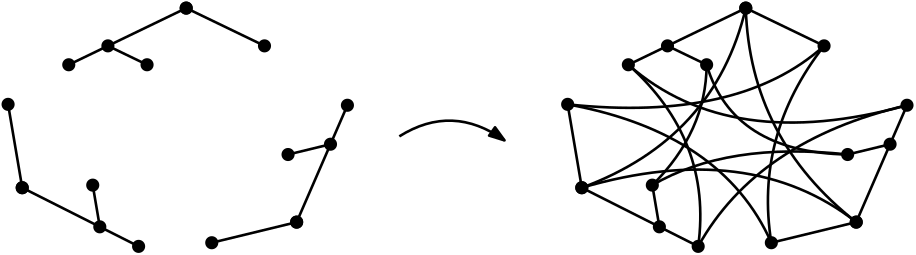}
    \caption{Construction of $\mathscr{T}(T)$}
    \label{fig:trigtree}
\end{figure}

See Figure \ref{fig:trigtree} for an example of the construction, where we choose $S$ consists of all vertices of degree at most $2$.  To instead obtain an at-most-trivalent graph, one could let $S$ be the set of leaves.

\begin{corollary}
If $T$ is a tree with at least two vertices, then $\textrm{gon}(T\boxempty K_3)=3$.
\end{corollary}

\begin{proof} Choosing $S=V(T)$ and performing our construction yields $\mathscr{T}(T)=T\boxempty K_3$, which thus has gonality~$3$.
\end{proof}

We can extend our construction to include certain multigraphs. Notice that we can add arbitrary edges between corresponding triples of vertices (which are already connected via a $3$-cycle) while retaining a graph of gonality $3$. This is because we still have the same non-degenerate harmonic morphism (the added edges are simply contracted) and because treewidth of a multigraph is equal to the treewidth of the underlying simple graph.

We can also generalize this construction somewhat to create graphs of gonality $k> 3$, although we are more constrained in what set of vertices we can choose for $S$.  Make $k$ copies of a tree $T$ that has at least two vertices. For each vertex $v$ of $T$ with $\text{val}(v)\leq k-1$, connect all the $k$ copies of $v$ to each other with $\binom{k}{2}$ edges. (Including some vertices with $\text{val}(v)\geq k$ is also allowable.)   Call the resulting graph $\mathscr{T}(T)$.  Our construction guarantees that each vertex has valence at least $k$, so $\textrm{gon}(\mathscr{T}(T))\geq k$  by Lemma \ref{lemma:minvalence}.  There is a natural harmonic morphism of degree $k$ from $\mathscr{T}(T)$ to $T$, which by the argument from $(2) \implies (1)$ in the proof of Theorem \ref{thm:multigraph} shows that $\textrm{gon}(\mathscr{T}(T))\leq k$. We conclude that  $\textrm{gon}(\mathscr{T}(T))= k$.


\section{Future Directions}

There are many cases of graphs of gonality $3$ that have not yet been covered by our results.  One could consider multigraphs that are not 3-edge-connected, and simple graphs that are neither 3-vertex-connected nor trivalent.  Moreover, all the results results in this paper only hold for combinatorial graphs, as opposed to metric graphs, which have lengths associated to their edges. A natural generalization of our work would be to determine the extent to which our results hold for metric graphs. The work by \cite{chan} on hyperelliptic graphs was done in the setting of metric graphs, so some of our results may extend via similar arguments.

Another natural question would be that of algorithmically testing whether or not a graph has gonality $3$.  In general, computing the divisorial gonality of a graph is NP-hard \cite{gij}, but it is possible to check if a graph has gonality $2$ in $O(n \log n + m)$ time \cite{recognize}.  The next step would be to develop an efficient algorithm for determining if a graph has gonality $3$.  There is a na\"{i}ve polynomial time algorithm that enumerates all effective divisors of degree $3$, then tests each such divisor against all possible placements of $-1$ chips using Dhar's burning algorithm.  However, a more efficient algorithm could be a helpful computational tool. The criteria we present in Theorem \ref{thm:maintheorem} may be useful for this endeavor.

\bigskip
 
\noindent \emph{Acknowledgements.}  The authors are grateful for support they received from NSF Grants DMS1659037 and DMS1347804 and from the Williams College SMALL REU program.  They also thank Melody Chan and two anonymous referees for many helpful comments on earlier versions of the paper.

\nocite{*}
\bibliographystyle{amsplain-ac}
\bibliography{bibliography}

\end{document}